\documentclass[12pt, amscd]{amsart}
\usepackage{amscd}
\usepackage{verbatim}
\usepackage[dvips]{color}

\input xy
\xyoption{all}

\usepackage{amssymb}

\DeclareMathAlphabet{\cat}{OT1}{cmss}{m}{sl}

\newtheorem{theorem}{Theorem}[section]

\newtheorem{proposition}[theorem]{Proposition}
\newtheorem{lemma}[theorem]{Lemma}
\newtheorem{corollary}[theorem]{Corollary}

\theoremstyle{definition}
\newtheorem{remark}[theorem]{Remark}

\newtheorem{example}[theorem]{Example}

\newcommand{\xra}{\xrightarrow}

\newcommand{\tens}{\otimes}

\newcommand{\gmu}{\boldsymbol{\mu}}

\renewcommand{\Im}{\operatorname{Im}}

\newcommand{\Ker}{\operatorname{Ker}}
\newcommand{\Coker}{\operatorname{Coker}}

\newcommand{\ind}{\operatorname{\hspace{0.3mm}ind}}
\newcommand{\ch}{\operatorname{char}}

\newcommand{\Br}{\operatorname{Br}}
\newcommand{\Spec}{\operatorname{Spec}}

\newcommand{\Gal}{\operatorname{Gal}}

\newcommand{\gPGL}{\operatorname{\mathbf{PGL}}}
\newcommand{\gSL}{\operatorname{\mathbf{SL}}}

\newcommand{\gSO}{\operatorname{\mathbf{SO}}}
\newcommand{\gGL}{\operatorname{\mathbf{GL}}}
\newcommand{\gm}{\operatorname{\mathbb{G}}_m}

\newcommand{\gme}{\operatorname{\mathbb{G}}_{m,E}}

\newcommand{\PGL}{\operatorname{PGL}}

\newcommand{\Int}{\operatorname{Int}}

\newcommand{\ed}{\operatorname{ed}}

\newcommand{\td}{\operatorname{tr.deg}}

\newcommand{\rank}{\operatorname{rank}}

\newcommand{\Z}{\mathbb{Z}}

\newcommand{\Q}{\mathbb{Q}}

\newcommand{\cT}{\mathcal T}

\newcommand{\cS}{\mathcal S}

\usepackage[hypertex]{hyperref}

\title[Essential dimension of simple algebras with involutions] 
{Essential dimension of simple algebras with involutions}

\author
[S. Baek] {Sanghoon Baek}

\address
{Department of Mathematics and Statistics, University of Ottawa, 585 King
Edward, Ottawa, ON K1N6N5, Canada}

\email {sbaek@uottawa.ca}

\thanks{The work has been supported by Beckenbach Dissertation Fellowship at the University of California at Los Angeles, Neher's NSERC Discovery grant 008836-2006, and Zainoulline's NSERC Discovery grant 385795-2010 and Accelerator Supplement grant 396100-2010.}

\begin{document}

\begin{abstract}
Let $1\leq m \leq n$ be integers with $m|n$ and $\cat{Alg}_{n,m}$ the class of central simple algebras of degree $n$ and exponent dividing $m$. In this paper, we find new, improved upper bounds for the essential dimension and $2$-dimension of $\cat{Alg}_{n,2}$. In particular, we show that $\ed_{2}(\cat{Alg}_{16,2})=24$ over a field $F$ of characteristic different from $2$.
\end{abstract}

\maketitle

\section{Introduction}

Let $\cT:\cat{Fields}/F\to\cat{Sets}$ be a functor (called an \emph{algebraic structure}) from the category $\cat{Fields}/F$ of field extensions over $F$ to the category $\cat{Sets}$ of sets. For instance, $\cT(E)$ with $E\in\cat{Fields}/F$ can be the sets of isomorphism classes of central simple $E$-algebras of degree $n$, \'etale $E$-algebras of rank $n$, quadratic forms over $E$ of dimension $n$, torsors (principal homogeneous spaces) over $E$ under a given algebraic group, etc. For fields $E,E'\in \cat{Fields}/F$, a field homomorphism $f:E\to E'$ over $F$ and $\alpha\in\cT(E)$, we write $\alpha_{E'}$ for the image of $\alpha$ under the morphism $\cT(f):\cT(E)\to \cT(E')$.

The notion of essential dimension was introduced by J. Buhler and Z. Reichstein
in \cite{BR97} and was generalized to algebraic structures by A. Merkurjev in \cite{BerhuyFavi03} and \cite{Merkurjev09}. The
essential dimension of an algebraic structure is defined to be the smallest number of
parameters needed to define the structure.

Let $E\in\cat{Fields}/F$ and $K\subset E$ a subfield over $F$. An
element $\alpha\in \cT(E)$ is said to be \emph{defined over $K$} and
$K$ is called a \emph{field of definition of $\alpha$} if there
exists an element $\beta\in \cT(K)$ such that $\beta_{E}=\alpha$. The \emph{essential dimension
of $\alpha$} is \[\ed(\alpha)=\min \{\td_{F}(K)\}\] over all fields of definition $K$ of $\alpha$. The
\emph{essential dimension of the functor $\cT$} is \[\ed(\cT)=\sup \{\ed(\alpha)\},\] where the supremum is taken over all fields $E\in\cat{Fields}/F$ and all
$\alpha\in \cT(E)$. Hence, the essential dimension of an algebraic structure $\cT$ measures the complexity of the structure in terms of the smallest number of parameters required to define the structure over a field extension of $F$.

Let $p$ be a prime integer. The \emph{essential $p$-dimension of  $\alpha$} is \[\ed_p(\alpha)=\min\{\ed(\alpha_{L})\},\] where $L$ ranges over all field extensions of $E$ of degree prime to $p$. In other words, $\ed_p(\alpha)=\min\{\td_{F}(K)\}$, where the minimum is taken over all field extensions $L/E$ of prime to $p$ and all subextensions $K/F$ of $L$ which are fields of definition of $\alpha_{L}$. Hence, $\ed(\alpha)\geq\ed_p(\alpha)$ and $\ed(\cT)\geq \ed_{p}(\cT)$ for all $p$. The \emph{essential $p$-dimension of $\cT$} is \[\ed_p(\cT)=\sup \{\ed_p(\alpha)\},\] where the supremum ranges over all fields $E\in\cat{Fields}/F$ and all $\alpha\in \cT(E)$.

Let $G$ be an algebraic group over $F$. The \emph{essential
dimension $\ed(G)$} (respectively, \emph{essential $p$-dimension $\ed_{p}(G)$}) of $G$ is defined to be $\ed(H^1(-,G))$ (respectively, $\ed_{p}(H^1(-,G))$), where $H^1(E,G)$ is the nonabelian cohomology set with respect to the finitely generated faithfully flat topology (equivalently, the set of isomorphism classes of $G$-torsors) over a field extension $E$ of $F$.

For every integer $n\geq 1$, a divisor $m$ of $n$ and any field extension $E/F$, let $\cat{Alg}_{n,m}(E)$ denote the set of isomorphism classes of central simple $E$-algebras of degree $n$
and exponent dividing $m$. Then, there is a natural bijection between $H^1(E,\gGL_{n}/\gmu_{m})$ and $\cat{Alg}_{n,m}(E)$ (see \cite[Example 1.1]{BM09}), thus \[\ed(\cat{Alg}_{n,m})=\ed(\gGL_{n}/\gmu_{m}) \text{ and } \ed_{p}(\cat{Alg}_{n,m})=\ed_{p}(\gGL_{n}/\gmu_{m}).\]

In this paper, we compute upper bounds for the essential dimension and $2$-dimension of $\cat{Alg}_{n,2}$. By a theorem of Albert, a central simple algebra has exponent dividing $2$ if and only if it admits an involution of the first kind (see \cite[Theorem 3.1]{Book}). Thus, any algebra $A$ in $\cat{Alg}_{n,2}(K)$ for any field extension $K/F$ has involutions of the first kind. Moreover, such $A$ has involutions of both symplectic and orthogonal types (see \cite[Corollary 2.8(2)]{Book}). By the primary decomposition theorem and \cite[Section 6]{BM10}, we have
\begin{equation}\label{primariydecompo}
\ed(\cat{Alg}_{n,2})=\ed(\cat{Alg}_{2^{r},2}) \text{ and } \ed_{2}(\cat{Alg}_{n,2})=\ed_{2}(\cat{Alg}_{2^{r},2}),
\end{equation}
where $2^{r}$ is the largest power of $2$ dividing $n$. Hence, we may assume that $n$ is a power of $2$.

By \cite[Remark 8.2 and Corollary 8.3]{BM10},
\[\ed_{2}(\cat{Alg}_{4,2})=\ed(\cat{Alg}_{4,2})=4 \text{ and } \ed_{2}(\cat{Alg}_{8,2})=\ed(\cat{Alg}_{8,2})=8\] over a field $F$ of $\ch(F)\neq 2$. By \cite[Theorem 1.1 and Theorem 1.2]{Baek}, \[\ed_{2}(\cat{Alg}_{4,2})=\ed(\cat{Alg}_{4,2})=3 \text{ and } 3\leq \ed_{2}(\cat{Alg}_{8,2})\leq \ed(\cat{Alg}_{8,2})\leq 10\] over a field $F$ of $\ch(F)=2$. In general, in  \cite[Theorem]{BM10}, the following bounds were established over a field $F$ of $\ch(F)\neq 2$:
\begin{equation}\label{genelbd}
(\log_{2}(n)-1)n/2\leq \ed_{2}(\cat{Alg}_{n,2}) \leq n^{2}/4+n/2 \text{ for all } n=2^r\geq 4.
\end{equation}
On the other hand, no upper bound for $\ed(\cat{Alg}_{n,2})$ was known for $n\geq 16$.

In the present paper, we find upper bounds for the essential dimension and $2$-dimension of $\cat{Alg}_{n,2}$ over an arbitrary field as follows:

\begin{theorem}\label{mainthm}
Let $F$ be an arbitrary field. Then, for any integers $n=2^r\geq 8$,
\begin{enumerate}
\item[(i)] $\ed(\gGL_{n}/\gmu_{2})=\ed(\cat{Alg}_{n,2})\leq (n-1)(n-2)/2$ \text{ if } $\ch{F}\neq 2$,
\item[(ii)] $\ed_{2}(\gGL_{n}/\gmu_{2})=\ed_{2}(\cat{Alg}_{n,2})\leq \begin{cases}
n^{2}/4 & \text{if $\ch{F}=2$},\\
n^{2}/16+n/2 & \text{if $\ch{F}\neq 2$}.
\end{cases}$
\end{enumerate}
\end{theorem}

We remark that our proof of the first inequality of (ii) in Theorem \ref{mainthm} holds for an arbitrary field.

\begin{corollary}\label{edslhighest16}
Let $F$ be a field of characteristic different from $2$. Then
\[\ed_{2}(\gGL_{16}/\gmu_{2})=\ed_{2}(\cat{Alg}_{16,2})=24.\]
\end{corollary}
\begin{proof}
The lower bound $24\leq \ed_{2}(\cat{Alg}_{16,2})$ follows from (\ref{genelbd}). On the other hand, the upper bound $\ed_{2}(\cat{Alg}_{16,2})\leq 24$ follows from Theorem \ref{mainthm}(ii).
\end{proof}

The paper is organized as follows. In Section \ref{generalst}, we introduce a general strategy to find upper bound for the essential dimension. In Section \ref{pfoffirstsedhalf}, we construct generically free representations for normalizers of maximal tori in $\gGL_{n}/\gmu_{2}$ and $\gSL_{n}/\gmu_{2}$, and a subgroup of the normalizer in $\gGL_{n}/\gmu_{2}$ to prove Theorem \ref{mainthm}(i) and the first part of (ii). The proof of the second part of (ii) in the main theorem is divided into two parts. In the first part, we study the structure of division algebras of exponent $2$ using involutions and then construct a certain surjective morphism $\Theta$. In the second part, we construct a generically free representation for a subgroup of $T_{r}\rtimes G_{r}$ to obtain the upper bound in the main theorem. In the last section, we relate the essential dimensions of $\cat{Alg}_{n,m}$ and of the split simple groups of type $A_{n-1}$.


To prove Theorem \ref{mainthm}(i) we use Lemmas \ref{lemmatrivial} and \ref{simpletnm}(ii). In fact, both Lemma \ref{simpletnm}(i) and (ii) give upper bounds for the essential dimension of $\gGL_{n}/\gmu_{2}$, but the latter provide a better bound combined with Lemma \ref{lemmatrivial}. Before we compute both (i) and (ii) in Lemma \ref{simpletnm}, it is not clear which one gives better bound. In particular, Lemma \ref{simpletnm}(ii) was obtained by Lemire in \cite{Lemire} under the assumption that the base field $F$ is an algebraically closed of characteristic $0$. Here we provide a different proof without the assumption on the base field $F$.

The general strategy (see Section \ref{generalst}) was first used by A.~Meyer and Z.~Reichstein in \cite{MR09} and \cite{MR10} to obtain upper bound for the essential $p$-dimension of $\gSL_{p^r}/\gmu_{p^r}$. In a subsequent paper \cite{Ruozzi} using similar general method, A.~Ruozzi obtained an improved upper bound. Based on his result, the upper bound in (\ref{genelbd}) was obtained. Here, the general strategy is further refined by methods of algebraic tori (Lemmas \ref{lemmabrau} and \ref{surjectlemm}), combined with the structure of simple algebras of exponent $2$ (Corollary \ref{etalesubalgebra}), and the results provide sharper bound (Theorem \ref{mainthm}(ii)).

\emph{Acknowledgements}: I am grateful to my advisor A.~Merkurjev for many useful discussions and support and to Z.~Reichstein for helpful comments. I also wish to thank E.~Neher and K.~Zainoulline for their reading of the manuscript. Section $3$ of this paper is based on the author's doctoral thesis at the University of California at Los Angeles.

\section{Preliminaries}\label{generalst}

A morphism of functors $\cS\to \cT$ from $\cat{Fields}/F$ to $\cat{Sets}$ is called \emph{surjective} if for any $E\in \cat{Fields}/F$, $\cS(E)\to \cT(E)$ is surjective. Such a surjective morphism gives an upper bound for the essential dimension of $\cT$,
\begin{equation}\label{ctcssss}
\ed(\cT)\leq \ed(\cS);
\end{equation}
see \cite[Lemma 1.9]{BerhuyFavi03}. Similarly, a morphism $\cS\to \cT$ from $\cat{Fields}/F$ to $\cat{Sets}$ is called \emph{$p$-surjective} if for any $E\in \cat{Fields}/F$ and any $\alpha\in \cT(E)$, there is a finite field extension $L/E$ of degree prime to $p$ such that $\alpha_{L}\in\Im(\cS(L)\to \cT(L))$. Then \begin{equation}\label{ctcsssssecond}
\ed_{p}(\cT)\leq \ed_{p}(\cS);
\end{equation}
see \cite[Proposition 1.3]{Merkurjev09}. Obviously, any surjective morphism is $p$-surjective for any prime $p$.

A field extension $K/F$ is called \emph{$p$-closed} (or \emph{$p$-special}) if every finite extension of $K$ is separable of degree a power of $p$. For a limit-preserving functor $\cT$ (e.g. the functor $H^{1}(-,G)$ for an algebraic group $G$ over $F$), by \cite[Lemma 3.3]{LMMR}, we have
\begin{equation}\label{limitpv}
\ed_{p}(\cT)=\ed_{p}(\cT_{K}),
\end{equation}
where $K/F$ is a $p$-closed field and $\cT_{K}$ is the restriction of $\cT$ to $\cat{Fields}/K$.

By (\ref{limitpv}), a surjective morphism $\cS\to \cT$ of limit-preserving functors over a $p$-closed field gives an upper bound for the essential $p$-dimension of $\cT$,
\[\ed_{p}(\cT)\leq \ed_{p}(\cS).\]In particular, if $\cT=H^{1}(-,G)$ for an algebraic group $G$ over $F$, then any \emph{generically free} representation, i.e., a finite dimensional vector space $V$ with $G$-action such that there exist a nonempty $G$-invariant open subset of $V$ on which $G$ acts freely, gives an upper bound for $\cT$,
\begin{equation}\label{verygeneral}
\ed_{p}(G)\leq\ed(G)\leq \dim(V)-\dim(G);
\end{equation}
see \cite[Theorem 3.4]{Re00} and \cite[Corollary 4.2]{Merkurjev09}. Moreover, if $G'$ is a closed subgroup of $G$ such that $([G:G'],p)=1$, then by \cite[Lemma 4.1]{MR09},
\begin{equation}\label{subgrouped}
\ed_{p}(G)=\ed_{p}(G').
\end{equation}

Let $T$ be a split torus over $F$ and $X$ be a finite set. Suppose that a finite group $H$ acts on the character group $T^{*}$ and $X$, and there is a $H$-equivariant homomorphism $\nu:\Z[X]\to T^{*}$. Then one can construct a generically free representation $V_{X}$ for $T\rtimes H$ as follows:
let \[V_{X}=\coprod_{x\in X}F\cdot x\] be a vector space over $F$ generated by $X$. The action of $H$ on $X$ induces a natural $H$-action on $V_{X}$. The split torus $T$ acts on $V_{X}$ by $t\cdot x=\nu(x)(t)\cdot x$ for $t\in T$. Therefore, the semidirect product $T\rtimes H$ acts linearly on $V_{X}$.

\begin{lemma}[{\cite[Lemma 3.3]{MR09}}]\label{semidirectgeneric}
The vector space $V_{X}$ gives a generically free representation for $T\rtimes H$ if and only if
\begin{enumerate}
\item[(i)] $\nu$ is surjective and
\item[(ii)] $H$ acts faithfully on $\Ker(\nu)$.
\end{enumerate}
Moreover, if these conditions are satisfied, then $\ed(T\rtimes H)\leq |X|-\rank(T^{*})$.
\end{lemma}

The last statement in Lemma \ref{semidirectgeneric} follows from (\ref{verygeneral}).

\section{Proofs of {\rm(i)} and the first part of {\rm(ii)} in Theorem {\rm\ref{mainthm}}}\label{pfoffirstsedhalf}
Let $G$ be a smooth reductive algebraic group over $F$. Let $T$ be a maximal torus of $G$ and $N$ the normalizer of $T$ in $G$. Then the canonical map
\[H^{1}(K,N)\to H^{1}(K,G)\]
is surjective for any field extension $K/F$ by \cite[Corollary 5.3]{CGR}, i.e., the morphism $H^{1}(-,N)\to H^{1}(-,G)$ is surjective and $p$-surjective for any prime $p$. Therefore, by (\ref{ctcssss}) and (\ref{ctcsssssecond}), we have
\begin{equation}\label{reductivenormalizer}
\ed(G)\leq \ed(N) \text{ and } \ed_{p}(G)\leq \ed_{p}(N).
\end{equation}

For any integer $n$ with $2|n$, consider the reductive group $\gGL_n/\gmu_{2}$ and the maximal torus $T_{n,2}:=\gm^{n}/\gmu_{2}$ in the group. Similarly, we let $T'_{n,2}$ be the maximal torus in $\gSL_{n}/\gmu_{2}$, i.e., $T'_{n,2}=T_{n,2}\cap (\gSL_{n}/\gmu_{2})$. In the following lemma we compute upper bounds for the essential dimension of normalizers of maximal tori in $\gGL_n/\gmu_{2}$ and $\gSL_{n}/\gmu_{2}$. We note that a different proof of Lemma \ref{simpletnm}(ii) is given in \cite[Proposition 5.6]{Lemire} under the assumption that the base field $F$ is an algebraically closed of characteristic $0$.

\begin{lemma}\label{simpletnm}
Let $F$ be an arbitrary field and $S_{n}$ the symmetric group on $n$ elements. Then we have
\begin{enumerate}
\item[(i)] $\ed(T_{n,2}\rtimes S_{n})\leq (n^{2}-n)/2$ for $n\geq 3$,
\item[(ii)] $\ed(T'_{n,2}\rtimes S_{n})\leq (n^{2}-3n+2)/2$ for $n\geq 6$.
\end{enumerate}
\end{lemma}
\begin{proof}
(i) Note that the character group $(T_{n,2})^{*}$ is isomorphic to
\[\{(t_{1},\cdots,t_{n})\in \Z^{n}|\text{ }t_{1}+\cdots+t_{n}=0 \text{ in } \Z/2\Z\}.\]

Let $e_{i,j}=(0,\cdots, 1,\cdots, -1, 0)$ be an element of $(T_{n,2})^{*}$, where $1$ and $-1$ are placed in the $i$th and $j$th positions respectively for $1\leq i\neq j\leq n$ and $0$'s are placed in other positions. Similarly, let $f_{i,j}=(0,\cdots, 1,\cdots, 1, 0)$, where $1$'s are placed in the $i$th and $j$th positions for $1\leq i\neq j\leq n$ and $0$'s are placed in other positions and let $g_{k}=(0,\cdots, -2,\cdots, 0)$ as an element of $(T_{n,2})^{*}$, where $-2$ is placed in the $k$th position for $1\leq k\leq n$ and $0$'s are placed in other positions.

Let $X$ be a set consisting of $f_{i,j}$ and $g_{k}$ for all $1\leq i\neq j \leq n$ and all $1\leq k \leq n$. Then $X$ is a $S_n$-invariant subset of $(T_{n,2})^{*}$ and $|X|=|f_{i,j}|+|g_{k}|=(n^2-n)/2+n$.

It is clear that $e_{i,j}$ and $f_{i,j}$ generate $(T_{n,2})^{*}$ as an abelian group, as the indices $i$ and $j$ run over $1$ to $n$. Since $f_{i,j}+g_{j}=e_{i,j}$, $X$ generates $(T_{n,2})^{*}$ as an abelian group and hence we have a surjective $S_{n}$-equivariant homomorphism $\nu:\Z[X]\to (T_{n,2})^{*}$ taking $f_{i,j}$ and $g_{k}$ to themselves.

We show that $S_{n}$ acts faithfully on $\Ker(\nu)$. Let $\sigma$ be a nontrivial element of $S_{n}$. Then there exists $1\leq i_{0} \leq n$ such that $\sigma(i_{0})\neq i_{0}$. Choose a $1\leq j_{0}\leq n$ which is different from $\sigma(i_{0})$ and $i_{0}$. Then $\sigma$ does not fix $2f_{i_{0},j_{0}}+g_{i_{0}}+g_{j_{0}}\in \Ker(\nu)$. Hence, by Lemma \ref{semidirectgeneric}, there exists a generically free representation for $T_{n,2}\rtimes S_{n}$, therefore, \[\ed(T_{n,2}\rtimes S_{n})\leq (n^2-n)/2+n-\rank((T_{n,2})^{*})=(n^2-n)/2.\]

\noindent
(ii) Similarly, the character group $(T'_{n,2})^{*}$ is isomorphic to 
\[\{(t_{1},\cdots,t_{n})\in \Z^{n}/\Z|\text{ }t_{1}+\cdots+t_{n}=0 \text{ in } \Z/2\Z\}.\]

Let $\overline{e}_{i,j}$ and $\overline{f}_{i,j}$ denote the classes of $e_{i,j}$ and $f_{i,j}$ in $\Z^{n}/\Z$. Let $X'$ be a set consisting of $\overline{f}_{i,j}$ for all $1\leq i\neq j \leq n$. Then $X'$ is a $S_n$-invariant subset of $(T'_{n,2})^{*}$ and $|X'|=|\overline{f}_{i,j}|=(n^2-n)/2$.

Note that for any $i\neq j$ there exists $k\neq i,j$ such that $\overline{e}_{i,j}=\overline{f}_{i,k}+\sum \overline{f}_{p,q}$, where the sum ranges over all different $p$ and $q$ with $p, q\neq j, k$. Therefore, $X'$ generates $(T'_{n,2})^{*}$ as an abelian group as $\overline{e}_{i,j}$ and $\overline{f}_{i,j}$ generate $(T'_{n,2})^{*}$. Hence we have a surjective $S_{n}$-equivariant homomorphism $\nu':\Z[X']\to (T'_{n,2})^{*}$ taking $\overline{f}_{i,j}$ to itself.

Recall that any nontrivial element $\sigma'$ of $S_{n}$ can be written as a product of disjoint cycles. Suppose that $\sigma'$ has a cycle which has length at least $3$, say $(ijk\cdots)$. Then $\sigma'(\overline{f}_{i,j}+\sum \overline{f}_{p,q})=\overline{f}_{j,k}+\cdots$, where the sum ranges over all different $p$ and $q$ with $p, q\neq i, j$. Hence $\sigma'$ acts faithfully on $\Ker(\nu')$. Now assume that $\sigma'$ consists of disjoint transpositions, say $(ij)(kl)(rs)\cdots$. Then $\sigma'(\overline{f}_{i,k}+\overline{f}_{j,r}+\sum \overline{f}_{p,q})=\overline{f}_{j,l}+\overline{f}_{i,s}+\cdots$, where the sum ranges over all different $p$ and $q$ with $p, q\neq i, j, k, r$. Therefore, $S_{n}$ acts faithfully on $\Ker(\nu')$ with the previous case. Hence, by Lemma \ref{semidirectgeneric}, there exists a generically free representation for $T'_{n,2}\rtimes S_{n}$, therefore, \[\ed(T'_{n,2}\rtimes S_{n})\leq (n^2-n)/2-\rank((T'_{n,2})^{*})=(n^{2}-3n+2)/2.\]

\end{proof}

\bigskip
\noindent
{\it Proof of Theorem {\rm\ref{mainthm}(i)}.}\quad By (\ref{reductivenormalizer}), Lemma \ref{simpletnm}(ii) and Lemma \ref{lemmatrivial}, we have
\[\ed(\cat{Alg}_{n,2})\leq \ed(\gSL_{n}/\gmu_{2})\leq \ed(T'_{n,2}\rtimes S_{n})\leq (n^{2}-3n+2)/2.\] 
\qed
\bigskip

Let $P_{n}$ be a Sylow $2$-subgroup of the symmetric group $S_{n}$ on $n$ elements. In the following lemma, we compute an upper bound for the essential dimension of $T_{2^r,2}\rtimes P_{2^r}$.

\begin{lemma}\label{torsioncsa}
Let $F$ be an arbitrary field. Then for any $r\geq 2$, we have
\[
\ed(T_{2^r,2}\rtimes P_{2^r})\leq 2^{2r-2},
\]
where $P_{2^r}$ is a Sylow $2$-subgroup of $S_{2^r}$.
\end{lemma}

\begin{proof}
Note that a Sylow $2$-subgroup $P_{2^r}$ of $S_{2^r}$ is isomorphic to $(P_{2^{r-1}})^{2}\rtimes \Z/2\Z$. Consider the $e_{i,j}$, $f_{i,j}$ and $g_{k}$ as in the proof of Lemma \ref{simpletnm}. We divide the set of integers $\{1, 2,\cdots, 2^r\}$ into two subsets $\Lambda_{1}:=\{1, 2,\cdots, 2^{r-1}\}$ and $\Lambda_{2}:=\{2^{r-1}+1, 2^{r-1}+2,\cdots, 2^r\}$. Let $X$ be a set consisting of $f_{i,j}$ and $g_{k}$ for all $1\leq i\neq j \leq 2^r$ such that $i$ and $j$ are placed in different $\Lambda_{l}$'s and all $1\leq k \leq 2^{r}$, where $l$ is either $0$ or $1$. Then $X$ is a $P_{2^r}$-invariant subset of $(T_{2^r,2})^{*}$ and $|X|=2^{2r-2}+2^r$.

It is clear that $e_{i,j}$ and $f_{i,j}$ generate $(T_{2^r,2})^{*}$ as an abelian group, as the indices $i$ and $j$ run over $1$ to $2^r$. Note that $f_{i,j}=f_{i,k}+f_{j,k}+g_{k}$ for all $i$ and $j$ which are in the same $\Lambda_{l}$'s, where $l$ is either $1$ or $2$. As \[e_{i,j}=\begin{cases}
f_{i,j}+g_{j} & \text{if $i$ and $j$ are in different $\Lambda_{l}$'s},\\
f_{i,k}+f_{j,k}+g_{j}+g_{k} & \text{otherwise,}\\
\end{cases}
\]
$X$ generates $(T_{2^r,2})^{*}$ as an abelian group and hence we have a surjective $P_{2^r}$-equivariant homomorphism $\nu:\Z[X]\to (T_{2^r,2})^{*}$ taking $f_{i,j}$ and $g_{k}$ to themselves.

We show that $P_{2^r}$ acts faithfully on $\Ker(\nu)$. Note that the center of $P_{2^r}$, which is generated by $\sigma:=(1, 2)(3, 4)\cdots (2^{r}-1, 2^r)$ and it is enough to show that $\sigma$ acts faithfully on $\Ker(\nu)$. In fact, $\sigma$ does not fix the non-zero element $2f_{1,2^{r-1}+1}+g_{1}+g_{2^{r-1}+1}\in \Z[X]$. Hence, by Lemma \ref{semidirectgeneric}, there exists a generically free representation for $T_{2^r,2}\rtimes P_{2^r}$, therefore, \[\ed(T_{2^r,2}\rtimes P_{2^r})\leq 2^{2r-2}+2^{r}-\rank((T_{2^r,2})^{*})=2^{2r-2}.\]
\end{proof}

\bigskip
\noindent
{\it Proof of the first part in Theorem {\rm\ref{mainthm}(ii)}.}\quad
As $(2,[T_{2^r,2}\rtimes S_{2^r}:T_{2^r,2}\rtimes P_{2^r}])=1$, we have $\ed_{2}(T_{2^r,2}\rtimes S_{2^r})=\ed_{2}(T_{2^r,2}\rtimes P_{2^r})$ by (\ref{subgrouped}). Therefore, by Lemma \ref{torsioncsa} and (\ref{reductivenormalizer}),
\[\ed_{2}(\gGL_{2^{r}}/\gmu_{2})\leq \ed_{2}(T_{2^r,2}\rtimes S_{2^r})=\ed_{2}(T_{2^r,2}\rtimes P_{2^r})\leq \ed(T_{2^r,2}\rtimes P_{2^r})\leq 2^{2r-2}.\]\qed

\section{Algebras with involutions}\label{alginvillll}

Let $A$ be a central simple algebra over $F$. For any $a\in A^{\times}$, we denote the inner automorphism of $A$ by $\Int(a)$: $\Int(a)(x)=axa^{-1}$ for all $x\in A$. For any subalgebra $B$ of $A$, we write $C_{A}(B)$ for the centralizer of $B$ in $A$. The following lemma characterizes all involutions of the first kind on $A$.

\begin{lemma}[{\cite[Proposition 2.7]{Book}}]\label{involdeg2}
Let $F$ be a field of $\ch(F)\neq 2$, $A$ be a central simple algebra over $F$ and $\sigma$ be an involution of the first kind on $A$. Then every involution $\sigma'$ of the first kind on $A$ is of the form $\Int(a)\circ \sigma$ for some $a\in A^{\times}$ uniquely determined up to a factor in $F^{\times}$, such that $\sigma(a)=\pm a$. Moreover, $\sigma$ and $\sigma'$ are of the same type if and only if $\sigma(a)=a$.
\end{lemma}

We use the following lemma for extension of involutions.

\begin{lemma}[{\cite[Theorem 4.14]{Book}}]\label{extinvol}
Let $F$ be a field of $\ch(F)\neq 2$, $A$ be a central simple algebra over $F$ with an involution $\sigma$ of the first kind, and $B$ be a simple subalgebra of $A$ with an involution $\tau$ such that $\tau |_{F}=\sigma |_{F}$. Then $A$ has involutions of both types whose restriction to $B$ is $\tau$, unless $\tau$ is of the first kind and the degree $\deg(C_{A}(B))$ is odd.
\end{lemma}

From now on until the end of this section, we assume that the base field $F$ is $2$-closed (i.e., every finite extension of $F$ is separable of degree a power of $2$) and is of characteristic different from $2$.

\begin{proposition}\label{mainproposition}
Let $r\geq 3$ be an integer, $F$ a $2$-closed field such that $\ch(F)\neq 2$ and $D$ a division $F$-algebra of degree $2^{r}$ and exponent $2$. Then for any biquadratic field extension $K_{1}K_{2}/F$ in $D$ with quadratic field extensions $K_{1}/F$ and $K_{2}/F$ there exists a quadratic extension $K_{3}/F$ in $D$ such that $K_{1}K_{2}K_{3}/F$ is a triquadratic extension in $D$.
\end{proposition}
\begin{proof}
By \cite[Theorem 3.1(1)]{Book}, $D$ has an involution of the first kind $\sigma$. Let $\tau_{1}$ and $\tau_{2}$ be two distinct nontrivial automorphisms of the field $K_{1}K_{2}$. As $\sigma|_{F}=\tau_{i}|_{F}$ for any $i=1,2$, there are two distinct involutions $\sigma_{1}$ and $\sigma_{2}$ of the same type on $A$ such that $\sigma_{i}|_{K_{1}K_{2}}=\tau_{i}$ by Lemma \ref{extinvol}.

By Lemma \ref{involdeg2}, there exists $d\in D^{\times}$ such that $\sigma_{1}=\Int(d)\circ \sigma_2$ and $\sigma_{i}(d)=d$ for all $i=1,2$. In particular, $d^{2}$ commutes with $K_{1}$ and $K_{2}$ and $F(d^{2})\cap K_{1}K_{2}=F$. If $F(d^{2})\neq F$, then $F(d^{2})$ contains a quadratic extension $K_{3}$ over $F$ by \cite[Proposition 1.1]{RowSalt}. Hence we have a triquadratic extension $K_{1}K_{2}K_{3}$ in $D$.

Suppose that $d^{2}\in F$. Then there exist quaternion subalgebras $Q_{1}:=(K_{1},d^{2})$ and $Q_{2}:=(K_{2},d^{2})$ of $D$. As the index $\ind(C_{D}(Q_{1}\tens Q_{2}))$ is bigger than or equal to $2$, $C_{D}(Q_{1}\tens Q_{2})$ contains a quadratic extension $K_{3}/F$ by \cite[Proposition 1.1]{RowSalt}. Therefore, we have \[K_{1}K_{2}K_{3}=K_{1}\tens K_{2}\tens K_{3}\subset Q_{1}\tens Q_{2}\tens C_{D}(Q_{1}\tens Q_{2})=D.\]
\end{proof}

\begin{corollary}\label{etalesubalgebra}
Let $r\geq 3$ be an integer and $F$ be a $2$-closed field such that $\ch(F)\neq 2$. Then for any division $F$-algebra $D$ of degree $2^{r}$ and exponent $2$ and an \'etale subalgebra $K_{1}K_{2}:=K_{1}\tens K_{2}$ of $D$ such that $\dim_{F}(K_{i})=2$ for $i=1,2$, there exists a maximal \'etale subalgebra $K_{1}K_{2}K:=K_{1}\tens K_{2}\tens K$ of $D$ with $\dim_{F}(K)=2^{r-2}$.
\end{corollary}
\begin{proof}
By Proposition \ref{mainproposition}, there exists a triquadratic field extension $K_{1}K_{2}K_{3}$ over $F$. Induction on $r$. If $r=3$, then $K=K_{3}$ satisfies the conclusion of corollary. For $r\geq 3$, the centralizer $C_{D}(K_{3})$ is a division $K_{3}$-algebra of degree $2^{r-1}$. By the induction hypothesis with $K_{1}K_{3}/K_{3}$ and $K_{2}K_{3}/K_{3}$, $C_{D}(K_{3})$ contains a subfield $K/F$ with $[K:K_{3}]=2^{r-3}$. Hence $D$ contains a field extension $K_{1}K_{2}K$ over $F$ such that $\dim_{F}(K)=2^{r-3}\cdot 2$.
\end{proof}

\section{Proof of the second part in Theorem {\rm\ref{mainthm}(ii)}}\label{ed2sec}

Let $n\geq 2$ be an integer, $G$ a subgroup of $S_{n}$ and $X$ a $G$-set of $n$ elements ($G$ acts on $X$ by permutations). For any divisor $m$ of $n$, we consider the surjective $G$-modules homomorphism $\bar\varepsilon:\Z[X]\to \Z/m\Z$, defined by $\bar\varepsilon(x)=\varepsilon(x)+m\Z$, where $\varepsilon: \Z[X]\to\Z$ is the \emph{augmentation homomorphism} given by $\varepsilon(x)=1$ for all
$x\in X$. Set $J=\Ker(\bar\varepsilon)$. Then we have an exact sequence
\begin{equation}\label{JXsequence}
0\to J\to \Z[X]\xra{\bar{\varepsilon}}\ \Z/m\Z \to 0.
\end{equation}
We shall need the following lemma (see also the proof of \cite[Theorem
8.1]{BM10}):

\begin{lemma}\label{lemmabrau}
Let $F$ be a field of $\ch(F)\mathrel{\not |}n$ and $T=\Spec F[J]$
be the split torus with the character group $J$. Then \[H^{1}(F,
T\rtimes G)=\coprod_{\Gal(E/F)=G}\Br_{m}(E/F),\] where the disjoint
union is taken over all isomorphism classes of Galois $G$-algebras
$E/F$.
\end{lemma}
\begin{proof}
Let $T_{\gamma}$ (respectively, $G_{\gamma}$) be the twist of $T$
(respectively, $G$) by the $1$-cocycle $\gamma\in Z^{1}(F, G)$.
Then by \cite[Proposition 28.11]{Book}, there is a natural bijection
between the fiber of $H^{1}(F, T\rtimes G)\to H^{1}(F,G)$ over
$[\gamma]$ and the orbit set of the group $G_\gamma(F)$ in $H^{1}(F,
T_{\gamma})$, i.e.,
\begin{equation}\label{cop}
H^{1}(F, T\rtimes G)\simeq\coprod H^{1}(F, T_{\gamma})/G_\gamma(F),
\end{equation}
where the coproduct is taken over all $[\gamma]\in H^{1}(F, G)$.

Let $E$ be the Galois $G$-algebra over $F$ associated to
$\gamma$. From (\ref{JXsequence}), we have the
corresponding exact sequence of algebraic groups
\[1\to \gmu_{m}\to \gm^{n} \to T\to 1\] and then the exact sequence
\begin{equation}\label{ccoopp}
1\to \gmu_{m}\to R_{E/F}(\gme)\to T_{\gamma}\to 1,
\end{equation}
each term of which is twisted by $\gamma$.
The exact sequence (\ref{ccoopp}) induces an exact sequence of Galois cohomology
\begin{equation}\label{cohomolnothing}
1\to H^{1}(F, T_{\gamma})\to H^{2}(F,\gmu_{m})=\Br_{m}(F)\to
H^{2}(E,\gme)=\Br(E)
\end{equation} by Eckmann-Faddeev-Shapiro's lemma and Hilbert's $90$.
The $G$-action on $R_{E/F}(\gme)$ restricts to the trivial action on
the subgroup $\gmu_{m}$. Let $\sigma\in G_\gamma(F)$ acts on
$T_{\gamma}=R_{E/F}(\gme)/\gmu_{m}$. The action of $\sigma$ and
(\ref{cohomolnothing}) induce the following diagram
\begin{equation*}
\xymatrix{
H^1(F,T_{\gamma}) \ar@{->}^{\sigma^{\ast}}[d]\ar@{^{(}->}[r] & H^{2}(F,\gmu_{m}) \ar@{=} [d]  \\
H^1(F,T_{\gamma})\ar@{^{(}->}[r]  & H^{2}(F,\gmu_{m}).\\
}
\end{equation*}
Therefore, $G_\gamma(F)$ acts trivially on $H^1(F,T_{\gamma})$, hence the result follows from (\ref{cop}).
\end{proof}

Let $r\geq 3$ be an integer. Let $G_{r}=S_{2}\times S_{2}\times S_{2^{r-2}}$ be a subgroup of the symmetric group $S_{2^{r}}$ on $2^{r}$ elements and let $H_{r}=S_{2}\times S_{2}\times P_{2^{r-2}}$ be a Sylow $2$-subgroup of $G_{r}$, where $P_{2^{r-2}}$ is a Sylow $2$-subgroup of $S_{2^{r-2}}$. Let $X_{r}$ be a $G_{r}$-set of $2^{r}$ elements ($G_{r}$ acts on $X_{r}$ by permutations). The action of $H_{r}$ may be described as follows: we subdivide the integers $1,2,\cdots, 2^{r}$ into four blocks $B_{1},B_{2},B_{3},B_{4}$ such that each block consists of $2^{r-2}$ consecutive integers. The $P_{2^{r-2}}$ permutes the elements of $B_{i}$ for all $1\leq i \leq 4$, $S_{2}$ interchanges $B_{2i-1}$ and $B_{2i}$ for all $i=1,2$, and another $S_{2}$ interchanges $B_{1}\cup B_{2}$ and $B_{3}\cup B_{4}$.

We set $J_{r}=\Ker(\Z[X_{r}]\xra{\bar{\varepsilon}} \Z/2\Z)$, where $\bar{\varepsilon}$ is the map with $m=2$ as in (\ref{JXsequence}). Applying Lemma \ref{lemmabrau} with $n=2^{r}$, $m=2$, $G=G_{r}$, $X=X_{r}$, $J=J_{r}$, and $T=T_{r}:=\Spec(F[J_{r}])$, we have a morphism \[\theta:H^{1}(-,T_{r}\rtimes G_{r}) \to \cat{Alg}_{2^{r},2}\] defined by $\theta(N)([A])=B$ for a field extension $N$ over $F$, where $[A]\in \Br_{2}(L/N)$ for some field extension $L/N$ with $\Gal(L/N)=G_{r}$ and $B$ is the central simple $N$-algebra of degree $2^{r}$ such that $[A]=[B]$ in $\Br_{2}(L/N)$.

Consider the morphism
\begin{equation}\label{thetaa}
\Theta:H^{1}(-,T_r\rtimes G_r) \coprod \big(\coprod_{1\leq i \leq r-1}\cat{Alg}_{2^{i},2}\big) \to \cat{Alg}_{2^{r},2}
\end{equation}
defined by \[[A]\mapsto \theta(N)([A]),\quad A_{i}\mapsto M_{2^{r-i}}(A_{i})\] over a field extension $N$ over $F$, where $A_{i}\in \cat{Alg}_{2^{i},2}(N)$ for $1\leq i \leq r-1$.

\begin{lemma}\label{surjectlemm}
If the base field $F$ is $2$-closed and is of characteristic different $2$, then $\Theta$ is surjective.
\end{lemma}
\begin{proof}
We show that $\Theta(N)$ is surjective for a field extension $N/F$. By the definition of $\Theta$, we only need to check the surjectivity for a division $N$-algebra $D$ of degree $2^r$ and exponent $2$. By \cite[Theorem 1.2]{RowSalt}, there exists an \'etale subalgebra $K_{1}K_{2}$ in $D$ such that $\dim_{N}(K_{i})=2$ for $i=1,2$. By Corollary \ref{etalesubalgebra}, there exists a maximal \'etale subalgebra $K_{1}K_{2}K$ in $D$ such that $\dim_{N}(K)=2^{r-2}$. Hence $\theta$ is surjective and so is $\Theta$.
\end{proof}

\begin{example}[see {\cite[Remark 3.10]{BM10}}]\label{firstexamp}
Let $r=3$. Then $G_3=H_3=S_{2}\times S_{2}\times S_{2}:=\langle\tau_{1}\rangle \times \langle\tau_{2}\rangle \times \langle\tau_{3}\rangle$. As the action of $H_3$ on $X_3$ is simply transitive, $X_3\simeq H_3$ as $H_3$-sets, hence $J_3$ is generated by $2$ and $\tau_{i}-1$ for $i=1,2,3$. Set $\Lambda_3:=\Z[H_3/\langle\tau_{1}\rangle]\oplus\Z[H_3/\langle\tau_{2}\rangle]\oplus\Z[H_3/\langle\tau_{3}\rangle]\oplus\Z[H_3/\langle\tau_{1}\tau_{2}\rangle]$. Define a map $\rho:\Lambda_3\to J_3$ by \[\rho(\overline{x_{1}},\overline{x_{2}},\overline{x_{3}},\overline{x_{4}})=\sum_{i=1}^{3}(\tau_{i}+1)x_{i}+(\tau_{1}\tau_{2}+1)x_{4}.\]
As $2=(\tau_1\tau_2+1)-\tau_1(\tau_2+1)+(\tau_1+1)$, $\rho$ is surjective. It is easy to check that $H_3$ acts on $\Ker(\rho)$ faithfully. Therefore, by Lemma \ref{semidirectgeneric}, $\ed_{2}(\cat{Alg}_{8,2})\leq 4+4+4+4-2^{3}=8$.

\end{example}

For an $x\in X_r$, let $H_{r,x}$ be the stabilizer of $x$ in $H_{r}=S_{2}\times S_{2}\times P_{2^{r-2}}:=\langle \tau_{1}\rangle \times \langle \tau_{2}\rangle \times P_{2^{r-2}}$. We set $P_{2^{r-2}}=(P_{2^{r-3}})^{2}\rtimes \langle \tau_{r}\rangle$.

\begin{lemma}\label{usss}
For any $r\geq 3$ and any $x\in X_r$, we have
\begin{enumerate}
  \item[(i)] $H_{r,x}\simeq H_{r-1,x}\times P_{2^{r-3}}$.
  \item[(ii)] The group $H_{r}$ is generated by $\tau_{1}, \tau_{2}, \tau_{r}$ and  $H_{r,x}$.
  \item[(iii)] The $\Z[H_{r}]$-module $J_r$ is generated by $2x, \tau_{1}x-x, \tau_{2}x-x$ and $\tau_{r}x-x$.
  \item[(iv)] $\tau_{r}H_{r,x}\tau_{r}\cap H_{r,x}\simeq H_{r-1,x}\times H_{r-1,x}$.
\end{enumerate}
\end{lemma}
\begin{proof}
 (i) By the action of $H_r$ (see page $8$), the stabilizer of $x$ in $H_{r}$ is the stabilizer of $x$ under the action of $1\times 1\times P_{2^{r-2}}\simeq P_{2^{r-2}}$ on the block $B_{i}$ containing $x$ for some $i$. As $P_{2^{r-2}}=(P_{2^{r-3}})^{2}\rtimes \langle \tau_{r}\rangle$, the stabilizer of $x$ in $P_{2^{r-2}}$ is $H_{r-1,x}\times P_{2^{r-3}}$.

\bigskip

(ii) Induction on $r$. The case $r=3$ comes from Example \ref{firstexamp}. By induction hypothesis we have $P_{2^{r-3}}=\langle \tau_{r-1}, H_{r-1,x}\rangle$. As $\tau_{r-1}$ is generated by $\tau_{r}$ and $P_{2^{r-3}}$, the result follows immediately.

\bigskip

(iii) As $H_{r}$ acts on $X_r$ transitively, the result follows from Lemma \ref{usss}(ii) and the sequence $(\ref{JXsequence})$.

\bigskip

(iv) As $\tau_{r}H_{r,x}\tau_{r}=H_{r,\tau_{r}(x)}$, the result follows from Lemma \ref{usss}(i).
\end{proof}

\bigskip
\noindent
{\it Proof of the second part in Theorem {\rm\ref{mainthm}(ii)}.}\quad For $r\geq 3$ and $x\in X_r$, we set \begin{align*}
\Lambda_{r} &:= \Z[H_{r}/\langle\tau_{1}\rangle\times H_{r,x}]\oplus \Z[H_{r}/\langle\tau_{2}\rangle\times H_{r,x}]\oplus\Z[H_{r}/\langle\tau_{1}\tau_{2}\rangle\times H_{r,x}] \\
& \oplus\Z[H_{r}/(\tau_{r}H_{r,x}\tau_{r}\cap H_{r,x})\rtimes \langle\tau_{r}\rangle].
\end{align*}
Define the map $\rho:\Lambda_{r}\to J_r$ by taking a generator of the first component (respectively, the second component) of $\Lambda_{r}$ to $\tau_{1}x+x$ (respectively, $\tau_{2}x+x$), a generator of the third component of $\Lambda_{r}$ to $\tau_{1}\tau_{2}x+x$, and a generator of the last component of $\Lambda_{r}$ to $\tau_{r}x+x$. By construction, this map is well defined. As $2x=(\tau_1\tau_2x+x)-\tau_1(\tau_2x+x)+(\tau_1x+x)$, $\rho$ is surjective by Lemma \ref{usss}(iii).

As $H_{r}$ acts faithfully on $\Ker(\rho)$ by Lemma \ref{faithful}, there exists a generically free representation for $T_r\rtimes H_r$ by Lemma \ref{semidirectgeneric}. Therefore, by Lemma \ref{semidirectgeneric} and (\ref{limitpv}), we have
\begin{align*}
\ed_{2}(T_r\rtimes H_r)&\leq \rank(\Z[H_{r}/\langle\tau_{1}\rangle\times H_{r,x}])+\rank(\Z[H_{r}/\langle\tau_{2}\rangle\times H_{r,x}])\\
& +\rank(\Z[H_{r}/\langle\tau_{1}\tau_{2}\rangle\times H_{r,x}])+\rank(\Z[H_{r}/(\tau_{r}H_{r,x}\tau_{r}\cap H_{r,x})\rtimes \langle\tau_{r}\rangle])\\
&-\rank(J_r)\\
& =2^{r-1}+2^{r-1}+2^{r-1}+2^{r+(r-1)-2-1}-2^{r} \text{ (by Lemma \ref{usss}(i),(iv)) }\\
& =2^{r-1}+2^{2r-4}.
\end{align*}

By (\ref{subgrouped}), $\ed_{2}(T_r\rtimes G_r)=\ed_{2}(T_r\rtimes H_r)$. As the morphism $\Theta$ in (\ref{thetaa}) is surjective by Lemma \ref{surjectlemm}, it follows from \cite[Lemma 1.10]{BerhuyFavi03} and (\ref{limitpv}) that \[\ed_{2}(\cat{Alg}_{2^{r},2})\leq \max\{\ed_{2}(T_r\rtimes G_r), \ed_{2}(\cat{Alg}_{2,2}), \cdots, \ed_{2}(\cat{Alg}_{2^{r-1},2})\}.\] By induction on $r$, we finally have \[\ed_{2}(\cat{Alg}_{2^{r},2})\leq \ed_{2}(T_r\rtimes G_r)\leq 2^{r-1}+2^{2r-4}.\]\qed

\bigskip

\begin{lemma}\label{faithful}
Let $\rho:\Lambda_{r}\to J_r$ be the morphism in the proof of the second part in Theorem {\rm\ref{mainthm}(ii)}. Then, the action of $H_{r}$ on $\Ker(\rho)$ is faithful.
\end{lemma}

\begin{proof}
It follows from the exact sequence (\ref{JXsequence}) that $J_r\tens \Q=\Q[X_r]$. Hence, by the exact sequence \[\Ker(\rho)\to \Lambda_{r}\xra{\rho} J_r,\] we have
\begin{align}\label{eeeE}
\Q[X_r]\oplus (\Ker(\rho))_{\Q} &=\Q[H_{r}/\langle\tau_{1}\rangle\times H_{r,x}]\oplus \Q[H_{r}/\langle\tau_{2}\rangle\times H_{r,x}]\oplus \Q[H_{r}/\langle\tau_{1}\tau_{2}\rangle\times H_{r,x}] \nonumber \\
& \oplus \Q[H_{r}/(\tau_{r}H_{r,x}\tau_{r}\cap H_{r,x})\rtimes \langle\tau_{r}\rangle].
\end{align}
By the actions of $\tau_{1}$ and $\tau_{2}$, the natural map \[i:\Z[X_r]\to \Z[X_r/\langle \tau_{1}\rangle]\oplus \Z[X_r/\langle \tau_{2}\rangle] \oplus \Z[X_r/\langle \tau_{1}\tau_{2}\rangle]\] is injective, hence we get the exact sequence \[0\to \Z[X_r]\xra{i} \Z[X_r/\langle \tau_{1}\rangle]\oplus \Z[X_r/\langle \tau_{2}\rangle] \oplus \Z[X_r/\langle \tau_{1}\tau_{2}\rangle]\to \Coker(i)\to 0\] and
\begin{equation}\label{faithfulexact}
\Q[X_r]\oplus (\Coker(i))_{\Q}=\Q[X_r/\langle \tau_{1}\rangle]\oplus \Q[X_r/\langle \tau_{2}\rangle] \oplus \Q[X_r/\langle \tau_{1}\tau_{2}\rangle].
\end{equation}
By (\ref{eeeE}) and (\ref{faithfulexact}), we have \[(\Ker(\rho))_{\Q}=(\Coker(i))_{\Q}\oplus \Q[H_{r}/(\tau_{r}H_{r,x}\tau_{r}\cap H_{r,x})\rtimes \langle\tau_{r}\rangle],\] thus it is enough to show that $H_{r}=\langle\tau_{1}\rangle \times \langle \tau_{2}\rangle \times [(P_{2^{r-3}})^{2}\rtimes \langle \tau_{r}\rangle]$ acts faithfully on $Y_r:=\Q[H_{r}/(\tau_{r}H_{r,x}\tau_{r}\cap H_{r,x})\rtimes \langle\tau_{r}\rangle]$. We prove this using case by case analysis.

\bigskip

{\it Case} $1$:\quad Suppose that $h\notin 1\times 1\times [(P_{2^{r-3}})^{2}\rtimes \langle \tau_{r}\rangle]\subset H_{r}$.
\begin{enumerate}
\item[(i)] If $h=\tau_{i}h'\in H_{r}$ for some $h'\in (P_{2^{r-3}})^{2}\rtimes \langle \tau_{r}\rangle$ and $i=1,2$, then $h\overline{\tau_{i}}=\overline{h'}\neq \overline{\tau_{i}}$ in $Y_r$.
\item[(ii)] If $h=\tau_{1}\tau_{2}h''\in H_{r}$ for some $h''\in (P_{2^{r-3}})^{2}\rtimes \langle \tau_{r}\rangle$, then $h\overline{\tau_{1}\tau_{2}}=\overline{h''}\neq \overline{\tau_{1}\tau_{2}}$ in $Y_r$.
\end{enumerate}

\bigskip

{\it Case} $2$:\quad Suppose now that $h\in 1\times 1\times [(P_{2^{r-3}})^{2}\rtimes \langle \tau_{r}\rangle]\subset H_{r}$. Recall from Lemma \ref{usss}(iv) that $\tau_{r}H_{r,x}\tau_{r}\cap H_{r,x}\simeq H_{r-1,x}\times H_{r-1,x}$. We view $(\tau_{r}H_{r,x}\tau_{r}\cap H_{r,x})\rtimes \langle\tau_{r}\rangle$ as a subgroup of $(P_{2^{r-3}})^{2}\rtimes \langle \tau_{r}\rangle$.
\begin{enumerate}
\item[(i)] Suppose that $h\in (P_{2^{r-3}})^{2}\rtimes \langle \tau_{r}\rangle\backslash (\tau_{r}H_{r,x}\tau_{r}\cap H_{r,x})\rtimes \langle\tau_{r}\rangle$. Then $h\overline{1}=\overline{h}\neq \overline{1}$ in $Y_r$.
\item[(ii)] Suppose that $h=h_{1}h_{2}\tau_{r}\in (H_{r-1,x}\times H_{r-1,x})\rtimes \langle\tau_{r}\rangle\simeq (\tau_{r}H_{r,x}\tau_{r}\cap H_{r,x})\rtimes \langle\tau_{r}\rangle$, where $h_{1}$ and $h_{2}$ are elements in the first and the second $H_{r-1,x}$, respectively. We may assume that $h_{1}\neq 1$. Choose a transposition $\delta\in P_{2^{r-3}}$ which does not fix $x$. Then $h\overline{\delta}=\overline{h_{1}h_{2}\tau_{r}\delta}=\overline{h_{1}h_{2}\delta'\tau_{r}}=\overline{\delta'h_{1}h_{2}\tau_{r}}=\overline{\delta'}\neq \overline{\delta}$, where $\delta'$ is the transposition such that $\tau_{r}\delta=\delta'\tau_{r}$.
\item[(iii)] Suppose that Let $h=h_{1}h_{2}\in H_{r-1,x}\times H_{r-1,x}\simeq \tau_{r}H_{r,x}\tau_{r}\cap H_{r,x}$, where $h_{1}$ and $h_{2}$ are elements in the first and the second $H_{r-1,x}$, respectively. We may assume that $h_{1}\neq 1$. Let $\tau_{r-1}$ be the permutation which acts on the same block with $h_{1}$. Then $h\overline{\tau_{r-1}}=\overline{h_{1}\tau_{r-1}}\neq \overline{\tau_{r-1}}$ in $Y_r$.
\end{enumerate}

\medskip
This completes the proof of faithfulness.
\end{proof}

\section{Essential dimension of split simple groups of type $A_{n-1}$}\label{typea}
Computing the essential dimension of split simple groups of type $A_{n-1}$, especially of their adjoint cases, is a long-standing open problem. It is wide open in general and the only completely known cases are the following:

\begin{example}\label{sl4282}
Let $F$ be an arbitrary base field.
\begin{enumerate}
\item[(i)] ($A_{1}$ and $A_{2}$) By \cite[Lemma 8.5.7]{Reichstein00a} and \cite[Lemma 9.4(c)]{Re00},
\[\ed_{2}(\gSL_{2}/\gmu_{2})=\ed(\gSL_{2}/\gmu_{2})=\ed_{3}(\gSL_{3}/\gmu_{3})=\ed(\gSL_{3}/\gmu_{3})=2.\]
\item[(ii)] ($A_{3}$) In \cite[Corollary 1.2]{Merkurjev09a}, the adjoint case was obtained as
\[\ed_{2}(\gSL_{4}/\gmu_{4})=\ed(\gSL_{4}/\gmu_{4})=5\]
over a field $F$ of $\ch(F)\neq 2$. As $A_{3}=D_{3}$, we have $\gSL_{4}/\gmu_{2}\simeq \gSO_{6}$. Therefore, by \cite[Theorem 10.4 and Example 12.7]{Re00},
\[\ed_{2}(\gSL_{4}/\gmu_{2})=\ed(\gSL_{4}/\gmu_{2})=5.\]
over a field $F$ of $\ch(F)\neq 2$.
Alternatively, one may use \cite[Theorem 8.13]{Reichstein00a} and Lemma \ref{lemmatrivial} for lower and upper bounds, respectively.
\end{enumerate}
\end{example}

In this section we relate the essential dimensions of $\cat{Alg}_{n,m}$ and of the split simple groups of type $A_{n-1}$ ($\simeq\gSL_{n}/\gmu_{m}$). As $\ed(\gSL_{n}/\gmu_{n})=\ed(\cat{Alg}_{n,n})$, the main purpose of the following lemma is the case where $m\neq n$.

\begin{lemma}\label{lemmatrivial}
Let $F$ be a field, $n\geq 1$ a integer not divisible by $\ch(F)$ and $m$ a divisor of $n$. Then \[\ed(\cat{Alg}_{n,m})\leq \ed(\gSL_{n}/\gmu_{m})\leq \ed(\cat{Alg}_{n,m})+1.\]
\end{lemma}
\begin{proof}
Consider the exact sequence
\begin{equation}\label{first}
1\to \gmu_{n/m}\xra{\alpha} \gSL_{n}/\gmu_{m}\to \gPGL_{n}\to 1.
\end{equation}
For a field extension $K/F$, the sequence (\ref{first}) induces an exact sequence
\[H^{1}(K,\gmu_{n/m})\to H^{1}(K,\gSL_{n}/\gmu_{m})\to H^{1}(K,\gPGL_{n})\xra{\partial} H^{2}(K,\gmu_{n/m}),\]
where $\partial: H^{1}(K,\gPGL_{n})\to H^{2}(K,\gmu_{n/m})=\Br_{n/m}(K)$ takes the isomorphism class of a central simple $K$-algebra $A$ of degree $n$ to the class $m[A]$. Therefore, we have the following sequence
\begin{equation}\label{second}
H^{1}(K,\gmu_{n/m})\xra{\alpha_{*}} H^{1}(K,\gSL_{n}/\gmu_{m})\twoheadrightarrow \cat{Alg}_{n,m}(K).
\end{equation}
As $H^{1}(K,\gSL_{n})$ is trivial and the first vertical map of the following diagram
\begin{equation*}
\xymatrix{
H^1(K,\gmu_{n}) \ar@{->}[d]\ar@{->}[r] & H^{1}(K,\gSL_{n}) \ar@{->}[d]  \\
H^1(K,\gmu_{n/m})\ar@{->}^{\alpha_{\ast}}[r]  & H^{1}(K,\gSL_{n}/\gmu_{m}).\\
}
\end{equation*}
is surjective, the image of $\alpha_{*}$ is trivial. Therefore, from (\ref{second}), we get a fibration of functors (see \cite[Definition 1.12]{BerhuyFavi03})
\[H^{1}(-,\gmu_{n/m})\rightsquigarrow H^{1}(-,\gSL_{n}/\gmu_{m})\twoheadrightarrow \cat{Alg}_{n,m}.\]
As $\ed(\gmu_{n/m})=1$, we get $\ed_{F}(\gSL_{n}/\gmu_{m})\leq \ed_{F}(\cat{Alg}_{n,m})+1$ by \cite[Proposition 1.13]{BerhuyFavi03}. The lower bound follows from (\ref{ctcssss}).
\end{proof}

\begin{remark}\label{rmklemmatrivial}
Let $F$ be an arbitrary base field.
\begin{enumerate}
\item[(i)] As $\gSL_{n}/\gmu_{m}$ is a subgroup of $\gGL_{n}/\gmu_{m}$ of codimension $1$, the second inequality of Lemma \ref{lemmatrivial} can be obtained by \cite[Theorem 6.19]{BerhuyFavi03}.
\item[(ii)] Recently, V.~Chernousov and A.~Merkurjev proved that
\begin{equation}\label{CMRecent}
\ed_{p}(\gSL_{p^r}/\gmu_{p^s})=\ed_{p}(\cat{Alg}_{p^r,p^s})+1
\end{equation}
for any $0\neq s< r$ over a field of $\ch(F)\neq p$ in \cite{CM}. Therefore, the computation of essential $p$-dimension of split simple group of type $A_{p^r-1}$ is reduced to the computation of $\ed_{p}(\cat{Alg}_{p^r,p^s})$. In particular, \cite[Corollary 8.3]{BM10}, we have $\ed_{2}(\cat{Alg}_{8,2})=\ed(\cat{Alg}_{8,2})=8$ over a field $F$ of $\ch(F)\neq 2$. Hence, by (\ref{CMRecent}) and Lemma \ref{lemmatrivial}, one can find \[\ed_{2}(\gSL_{8}/\gmu_{2})=\ed(\gSL_{8}/\gmu_{2})=9\]
over a field $F$ of $\ch(F)\neq 2$. Moreover, by (\ref{CMRecent}) and Corollary \ref{edslhighest16}, one can conclude that
\[\ed_{2}(\gSL_{16}/\gmu_{2})=25;\]
see \cite[Corollary 1.2]{CM}.
\end{enumerate}
\end{remark}

\end{document}